\documentclass[letterpaper, 10pt, conference]{ieeeconf}  
\IEEEoverridecommandlockouts                              
\usepackage[applemac]{inputenc}
\usepackage{graphics,bm} 
\usepackage{epsfig,color} 
\usepackage{epstopdf}
\usepackage{amsmath} 
\usepackage{amssymb}  
\usepackage{epstopdf,pmat}
\usepackage{graphicx}
\usepackage{caption}
\usepackage{subcaption}

\newtheorem{thm}{Theorem}
\newtheorem{defi}{Definition}
\newtheorem{lem}{Lemma}
\newtheorem{cor}{Corollary}
\newtheorem{rem}{Remark}
\newtheorem{prop}{Proposition}

\title{\LARGE \bf
A Convex Approach to Hydrodynamic Analysis
}

\author{Mohamadreza Ahmadi,  Giorgio Valmorbida, Antonis Papachristodoulou
\thanks{The authors are with the Department
of Engineering Science, University of Oxford, Oxford,
OX1 3PJ, UK e-mail: (\{mohamadreza.ahmadi, giorgio.valmorbida, antonis\}@eng.ox.ac.uk). M. Ahmadi is supported by the Clarendon Scholarship and the Sloane-Robinson Scholarship. G. Valmorbida is also affiliated to Somerville College, University of Oxford, Oxford, U.K. A. Papachristodoulou was supported in part by the Engineering and Physical Sciences Research Council projects EP/J012041/1, EP/I031944/1, EP/M002454/1,
and EP/J010537/1.
}
}

\begin{document}

\maketitle
\thispagestyle{empty}
\pagestyle{empty}

\begin{abstract}

We study stability and  input-state analysis of three dimensional (3D)~incompressible,  viscous flows with invariance in one direction. By taking advantage of this invariance property, we propose a class of Lyapunov and storage functionals. We then consider exponential stability, induced $\mathcal{L}^2$-norms, and input-to-state stability (ISS). For streamwise constant flows, we formulate conditions based on matrix inequalities. We show that in the case of polynomial laminar flow profiles the matrix inequalities can be checked via convex optimization. The proposed method is illustrated by an example of rotating Couette flow. 
\end{abstract}

\section{INTRODUCTION}

The dynamics of incompressible fluid flows is described by a set of nonlinear partial differential equations known as the Navier-Stokes equations. The properties of such flows are then characerized in terms of a dimensionless parameter $Re$ called the Reynolds number. Experiments show that many flows have a critical Reynolds number $Re_C$ below which global stability is ensured. However, spectrum analysis of the linearized Navier-Stokes equations, considering only infinitesimal perturbations, 
predicts a linear stability limit $Re_L$ which upper-bounds $Re_C$~\cite{DR81}. On the other hand, the bounds using energy methods $Re_E$, the limiting value for which the energy of arbitrary large perturbations decreases monotonically, are much below $Re_C$~\cite{J76}. For example, $Re_E =32.6$~\cite{Se59}, $Re_L = \infty$~\cite{Romanov73} and $Re_C=350$~\cite{TA92} for 3D Couette flow.


 The discrepancy between $Re_L$ and $Re_C$ have long been attributed to the eigenvalues analysis approach~\cite{Trefethen30071993}, citing a phenomenon called \textit{transient growth}; i.e., although the perturbations to the linearized Navier-Stokes equation are stable, they undergo high amplitude transient amplifications that steer the trajectories out of the region of linearization. This has led to studying the resolvent operator or $\varepsilon$-pseudospectra based on the general solution to the linearized Navier-Stokes equations~\cite{Sch07}. Another method for studying stability is based on spectral truncation of the  Navier-Stokes equations into an ODE system.  Recently in~\cite{GC12,CGHP14}, a method was proposed based on keeping  a number of modes from Galerkin expansion and bounding the energy of the remaining modes. However, these bounds on $Re_C$ turn out to be conservative.
   
    Since the seminal paper by Reynolds~\cite{Rey83}, it was observed that external excitations and body forces play an important role in flow instabilities.  Mechanisms such as energy amplification of external excitations  have shown to be crucial in understanding transition to turbulence~\cite{J76}.  Energy amplification  of stochastic forcings to the linearized Navier-Stokes equations in parallel channel flows was studied in~\cite{FI93,BD01}. In~\cite{BD01}, using the linearized Navier-Stokes equation, it was shown analytically, through the calculation of traces of operator Lyapunov equations, that the $\mathcal{H}^2$-norm from streamwise constant excitations to perturbation velocities is proportional to $Re^3$. The $O(Re^3)$ amplification mechanism of the linearized Navier-Stokes equation was verified in~\cite{JB05} and \cite{mjphd04}, where the influence of each component of the body forces was calculated in terms of $\mathcal{H}^2$ and $\mathcal{H}^\infty$-norms. Input-output analysis of a model of plane Couette flow was carried out in~\cite{GMBPD11} to study the nonlinear mechanisms associated with turbulence. In another vein, an input-state analysis method for the linearized Navier-Stokes equation by calculating the spatio-temporal impulse responses was given in~\cite{JB01}. 
   
 In this paper, we study the stability and input-state properties of incompressible, viscous fluid flows. We study input-state properties such as induced $\mathcal{L}^2$-norms from body forces to perturbation velocities and ISS. In particular, we consider flows with invariance in one of the three spatial coordinates. For such flows, we formulate a suitable structure as a Lyapunov/storage functional. Then, based on these functionals, for streamwise constant flows, we propose conditions based on matrix inequalities. In the case of polynomial laminar velocity profiles, e.g. Couette and Poiseuille flows, these inequalities can be checked via convex optimization using available computational tools. The proposed method is applied to the analysis problem of a rotating Couette flow.

The paper is organized as follows. The next section presents some preliminary results. In Section~\ref{sec:Lyap}, we formulate the Lyapunov/storage functional structure. Section~\ref{sec:convex} is concerned with the convex formulation for streamwise constant flows. The proposed method is illustrated by studying an example of a model of rotating Couette flow in Section~\ref{sec:NR}. Finally, Section~\ref{sec:conclusions} concludes the paper and provides directions for future research.\\


\noindent\textbf{Notation:} The $n$-dimensional Euclidean space is denoted by $\mathbb{R}^n$. 
The $n \times n$ identity matrix is denoted by $\mathrm{I}_{n \times n}$. A domain $\Omega \subset \mathbb{R}^n$ is a connected, open subset of $\mathbb{R}^n$, and $\overline{\Omega}$ is the closure of set $\Omega$. The boundary $\partial \Omega$ of set $\Omega$ is defined as $\overline{\Omega} \setminus \Omega$ with $\setminus$ denoting set subtraction. The space of $p$-th power integrable functions defined over $\Omega$ is denoted  $\mathcal{L}^p_\Omega$ endowed with the norm 
$
\| (\cdot) \|_{\mathcal{L}^p_\Omega} = \left( \iint_\Omega (\cdot)^p \,\, \mathrm{d}\Omega  \right)^{\frac{1}{p}},
$ 
for $1\le p < \infty$, and $
\| (\cdot) \|_{\mathcal{L}^\infty_\Omega} = \sup_{x \in \Omega} | (\cdot) |,
$ 
for $p = \infty$.
Also, we denote by $\mathcal{L}^2_{[t_0,T],\Omega}$, with $t_0\ge0$, the space of square integrable functions in $x\in \Omega$ and $t \in [t_0,T]$ with the norm
$$
\| (\cdot) \|_{\mathcal{L}^2_{[t_0,T),\Omega}} = \left( \int_{t_0}^T \| u \|_{\mathcal{L}^2_\Omega}^2\,\, \mathrm{d}t  \right)^{\frac{1}{2}}.
$$
The space of $k$-times continuous differentiable functions defined on $\Omega$ is denoted by $\mathcal{C}^k(\Omega)$.
  If $p\in \mathcal{C}^1$, then   $\partial_{x_1} p$ is used to denote the derivative of $p$ with respect to variable $x_1$, i.e. $\partial_{x_1} :=\frac{\partial}{\partial x_1}$. A continuous strictly increasing function $k: [0,a)^p \to \mathbb{R}_{\ge0}$, satisfying $k(0)=0$, belongs to class $\mathcal{K}$. If $a=\infty$ and $\lim_{x \to \infty} k(x) = \infty$, $k$ belongs to class $\mathcal{K_\infty}$.  The unit vector in direction $x_i$ is denoted by $\overrightarrow{e}_{i}$. For a scalar function ${v}$, $\nabla {v} = \sum_i \partial_{x_i} v \overrightarrow{e}_{i}$ denotes the gradient and $\nabla^2 {v} = \sum_i \partial_{i}^2 v$ denotes the Laplacian. For a vector valued function $\boldsymbol{w}=\sum_i w_i \overrightarrow{e}_{i}$, the divergence $\nabla \cdot \boldsymbol{w}$ is given by $\nabla \cdot \boldsymbol{w} = \sum_i \partial_{x_i} w_i$. 


\section{Preliminiaries}

\subsection{Flow Model}

We consider  incompressible, viscous flows with invariance in one of the directions\footnote{Invariance in one direction is  a common assumption in the case of  several fluid models, namely, Couette flow, Poiseuille flow, Taylor-Couette flow, etc.} $x_m$, $m \in \{1,2,3\}$, i.e.,~\mbox{$\partial_{x_m} =0$}. Let $I= \{1,2,3\}-\{m\}$. The flow dynamics is described by the Navier-Stokes equations, given~by
\begin{eqnarray} \label{eq:NS}
\partial_t \boldsymbol{\bar{u}} &=& \frac{1}{Re} \nabla^2  \boldsymbol{\bar{u}}-  \boldsymbol{\bar{u}}\cdot \nabla  \boldsymbol{\bar{u}}- \nabla {\bar{p}} + F  \boldsymbol{\bar{u}}+\boldsymbol{d}, \nonumber \\
0 &=& \nabla \cdot \bar{\boldsymbol{u}},
\end{eqnarray}

\noindent where $t>0$, $F \in \mathbb{R}^{3\times3}$, and~{$\mathrm{x} \in \Omega = \Omega_{i} \times \Omega_j \subset \mathbb{R} \times \mathbb{R}$} with $\mathrm{x}=(x_i,x_j)'$, $i,j \in I$ being the spatial coordinates. 
The dependent variable $\boldsymbol{d}(t,\mathrm{x}) = \begin{bmatrix} d_1(t,\mathrm{x}) & d_{2}(t,\mathrm{x})  & d_{3}(t,\mathrm{x})  \end{bmatrix}'$ is the input vector representing exogenous excitations or body forces, $\boldsymbol{\bar{u}}(t,\mathrm{x}) = \begin{bmatrix} \bar{u}_{1}(t,\mathrm{x}) & \bar{u}_{2}(t,\mathrm{x}) & \bar{u}_{3}(t,\mathrm{x})  \end{bmatrix}'$ is the velocity vector, and ${\bar{p}}(t,\mathrm{x})$ is the pressure. 

We consider perturbations $(\boldsymbol{u},{p})$ to the stationary flow $(\boldsymbol{U},{P})$. That is,
\begin{equation} \label{eq:perturbsub}
\boldsymbol{\bar{u}} = \boldsymbol{u} + \boldsymbol{U},~{\bar{p}} = {p} + {P},
\end{equation}
where $(\boldsymbol{U},P)$ satisfy
\begin{eqnarray} \label{eq:eU}
0 &=& \frac{1}{Re} \nabla^2 \boldsymbol{U} -\boldsymbol{U} \cdot \nabla \boldsymbol{U} - \nabla P + F\boldsymbol{U}, \nonumber \\
0 &=& \nabla \cdot \boldsymbol{U}.
\end{eqnarray}
Substituting~\eqref{eq:perturbsub} in~\eqref{eq:NS} and using~\eqref{eq:eU}, we  obtain the perturbation  dynamics
\begin{eqnarray} \label{eq:mainNS}
\partial_t \boldsymbol{u} &=& \frac{1}{Re} \nabla^2 \boldsymbol{u} - \boldsymbol{u} \cdot \nabla \boldsymbol{u} - \boldsymbol{U} \cdot \nabla \boldsymbol{u} - \boldsymbol{u} \cdot \nabla \boldsymbol{U} \nonumber \\
 && -\nabla p + F \boldsymbol{u} + \boldsymbol{d}, \nonumber \\
0 &=& \nabla \cdot \boldsymbol{u}.
\end{eqnarray}
%
In this paper, we concentrate on perturbations with no-slip boundary conditions $\boldsymbol{u}|_{\partial \Omega} \equiv 0$ and periodic boundary conditions.

\subsection{Stability and Input-to-State Analysis}

In this section, we briefly review a number of definitions and results from~\cite{VAP14} and~\cite{AVP14}.

\begin{defi} [Exponential Stability] \label{def:stab}
The stationary solution $(0,p_0)$ of~\eqref{eq:mainNS} with $\boldsymbol{d} \equiv 0$ is exponentially stable in $\mathcal{L}^2_\Omega$, if there exists a $\lambda >0$, such that for all $t \ge 0$
\begin{equation}
 \| \boldsymbol{u}(t,\mathrm{x}) \|^2_{\mathcal{L}^2_\Omega}  \le \| \boldsymbol{u}(0,\mathrm{x}) \|^2_{\mathcal{L}^2_\Omega} e^{-\lambda t}.
 \end{equation}
 That is, system~\eqref{eq:NS} converges to the laminar flow $(\boldsymbol{U},P)$ as in~\eqref{eq:eU}.
\end{defi}

\begin{defi} [input-to-State Properties]
\hfill
\begin{itemize}
\item [A.] \textit{Induced $\mathcal{L}^2$-norm Boundedness}: For some \mbox{$\eta_i >0$, $i=1,2,3$},
\begin{equation} \label{eq:L2}
\| \boldsymbol{u}(t,\mathrm{x}) \|_{\mathcal{L}^2_{[0,\infty),\Omega}} \le \sum_{i=1}^3\eta_i \| d_{i}(t,\mathrm{x}) \|_{\mathcal{L}^2_{[0,\infty),\Omega}}
\end{equation}
subject to zero initial conditions $\boldsymbol{u}(0,\mathrm{x})\equiv0,~\forall \mathrm{x} \in \Omega$.

\item [B.] \textit{Input-to-State Stability}: For some scalar $\psi>0$, functions $\beta,\tilde{\beta},\chi \in \mathcal{K}_\infty$, and  $\sigma \in \mathcal{K}$, it holds that
\begin{multline}\label{eq:iss}
\|\boldsymbol{u}(t,\mathrm{x})\|_{\mathcal{L}^2_\Omega} \le \beta \bigg( e^{-\psi t} \chi \left(\|\boldsymbol{u}(0,\mathrm{x})\|_{\mathcal{L}^2_\Omega} \right)  \bigg)  \\+ \tilde{\beta} \bigg( \sup_{\tau \in [0,t)} \big( \iint_\Omega \sigma \big(|\boldsymbol{d}(\tau,\mathrm{x})|\big) \,\, \mathrm{d}\Omega \big)  \bigg),
\end{multline}
\noindent for all $t>0$.
\end{itemize}
\end{defi}

 \begin{rem}
  Due to nonlinear dynamics, the actual induced $\mathcal{L}^2$-norms of system~\eqref{eq:mainNS} are nonlinear functions of $\|\boldsymbol{d}\|_{\mathcal{L}^2_\Omega}$. The quantities $\eta_i,~i=1,2,3$ provide  upper-bounds on the actual  induced $\mathcal{L}^2$-norms.
  \end{rem}

\begin{rem}
The ISS property~\eqref{eq:iss} implies the exponential convergence to the laminar flow $(\boldsymbol{U},P)$ in $\mathcal{L}^2_\Omega$ when $\boldsymbol{d} \equiv 0$. Moreover, as $t \to \infty$, we obtain
\begin{multline}
\lim_{t \to \infty}\| \boldsymbol{u}(t,\mathrm{x}) \|_{\mathcal{L}^2_\Omega} \le \beta \left(  \iint_\Omega \|\sigma(|\boldsymbol{d}(t,\mathrm{x})|) \|_{\mathcal{L}^\infty_{[0,\infty)} }\,\, \mathrm{d}\Omega  \right)  \\
 \le  \beta \left(   \iint_\Omega \sigma(\|\boldsymbol{d}(t,\mathrm{x})\|_{\mathcal{L}^\infty_{[0,\infty)} }) \,\, \mathrm{d}\Omega  \right),
\end{multline}
wherein, the fact that $\sigma, \beta \in \mathcal{K}$ is used. Hence, as long as the external excitations or body forces $\boldsymbol{d}$ are bounded in $\mathcal{L}^\infty_{[0,\infty)}$ (this encompasses persistent excitations), the perturbation velocities  $\boldsymbol{u}$ are bounded in  $\mathcal{L}^2_\Omega$ sense. 
\end{rem}

The next result converts the tests for exponential stability, induced $\mathcal{L}^2$-norm boundedness, and ISS into the existence problem of a Lyapunov or a storage functional satisfying a set of inequalities.

\begin{thm} \label{bigthm}
Consider  perturbation model~\eqref{eq:mainNS}. If there exist a positive definite Lyapunov functional $V(\boldsymbol{u})$ and a positive semidefinite  storage functional $S(\boldsymbol{u})$, positive scalars $\{\eta_i\}_{i\in \{1,2,3\}}$, $\{c_i\}_{i \in \{1,2,3\}}$, $\psi$, and functions $\beta_1,\beta_2 \in \mathcal{K}_\infty$, $\sigma \in \mathcal{K}$, such that\\
\mbox{I)} when $\boldsymbol{d} \equiv 0$,
\begin{equation} \label{sda1}
c_1 \| \boldsymbol{u} \|_{\mathcal{L}^2_\Omega}^2 \le V(\boldsymbol{u})  \le c_2 \| \boldsymbol{u} \|_{\mathcal{L}^2_\Omega}^2,
\end{equation}
\begin{equation} \label{sddsf1}
\partial_t V(\boldsymbol{u}) \le -c_3 \| \boldsymbol{u} \|_{\mathcal{L}^2_\Omega}^2,
\end{equation}
\mbox{II)}
\begin{eqnarray} \label{e6}
\partial_t S(\boldsymbol{u}) \le - \iint_\Omega \boldsymbol{u}^\prime \boldsymbol{u} \,\, \mathrm{d}\Omega
+  \iint_\Omega \boldsymbol{d}^\prime \left[ \begin{smallmatrix} \eta_1^2 & 0 & 0 \\ 0 & \eta_2^2 & 0 \\ 0 & 0 & \eta_3^2  \end{smallmatrix} \right] \boldsymbol{d} \,\, \mathrm{d}\Omega,
\end{eqnarray}
\mbox{III)}
\begin{equation} \label{e11}
\beta_1(\|\boldsymbol{u}\|_{\mathcal{L}^2_\Omega}) \le S(\boldsymbol{u}) \le \beta_2(\|\boldsymbol{u}\|_{\mathcal{L}^2_\Omega}),
\end{equation}
\begin{equation} \label{e12}
\partial_t S(\boldsymbol{u}) \le -  \psi S(\boldsymbol{u}) + \iint_\Omega \sigma(|\boldsymbol{d}(t,\mathrm{x})|) \,\, \mathrm{d}\Omega,
\end{equation}

\noindent for all $t>0$, then, respectively, system \eqref{eq:mainNS}

\noindent \mbox{I)}  is exponentially stable,

\noindent \mbox{II)} has  induced  \textit{$\mathcal{L}^2$-norm} upper-bounds  $\eta_i$, $i=1,2,3$ as in~\eqref{eq:L2},

\noindent \mbox{III)}  is \textit{ISS} and satisfies \eqref{eq:iss} with $\chi = \beta_2$, $\beta(\cdot)=\beta_1^{-1}\circ 2(\cdot)$ and \mbox{$\tilde{\beta}(\cdot)=\beta_1^{-1}\circ \frac{2}{\psi} (\cdot)$}.
\end{thm}

\begin{proof}
This is a direct application of Theorem 1 in \cite{VAP14} and Theorem 1 in \cite{AVP14}. 
\end{proof}

\section{Lyapunov and Storage Functionals\\  for Fluid Flows} \label{sec:Lyap}

In this section, we derive  classes of Lyapunov and storage functionals suitable for analysis of system~\eqref{eq:mainNS} subject to invariance in one of the three spatial coordinates. In the following, we adopt Einstein's multi-index notation over index $j$, that is the sum over repeated indices $j$, e.g., $v_j \partial_{x_j} u_j = \sum_j v_j\partial_{x_j} u_j $.

The perturbation model \eqref{eq:mainNS}  can be re-written as
\begin{eqnarray} \label{eq:mainNSEin}
\partial_t u_i &=& \frac{1}{Re} \nabla^2 u_i - u_j  \partial_{x_j} u_i - U_j  \partial_{x_j} u_i \nonumber \\
&& - u_j  \partial_{x_j} U_i - \partial_{x_i} p + F_{ij} u_j + d_i, \nonumber \\
0 &=& \partial_{x_j} u_j.
\end{eqnarray}
where $i,j \in \{1,2,3\}$ and $F_{ij}$ is the $(i,j)$ entry of $F$. 


The next theorem states, under which Lyapunov/storage functional structure, the time derivative of the Lyapunov/storage functional takes the form of a quadratic form in  dependent variables $\boldsymbol{u}$ and their spatial derivatives, by removing the nonlinear convection and pressure terms.

\begin{prop} \label{prop1}
Consider the perturbation model~\eqref{eq:mainNSEin} subject to  periodic or   no-slip boundary conditions $\boldsymbol{u}|_{\partial\Omega} =0$. Assume~\eqref{eq:mainNSEin} is invariant with respect to $x_m$, $m\in\{1,2,3\}$. Let $I = \{1,2,3\}-\{m\}$ and
\begin{multline} \label{eq:Lyap}
V(\boldsymbol{u}) = \frac{1}{2}\iint_\Omega \boldsymbol{u}^\prime \left[ \begin{smallmatrix} k_m & 0 & 0 \\ 0 & k_i & 0 \\ 0 & 0 & k_j \end{smallmatrix} \right] \boldsymbol{u} \,\, \mathrm{d}\Omega \\=\frac{1}{2}\iint_\Omega \sum_{i=1}^3k_i u_i(t,\mathrm{x})^2\,\, \mathrm{d}\Omega,
\end{multline}
where $k_i = k_j$ for $i,j \in I$, be a candidate Lyapunov or storage functional. Then, the time derivative of~\eqref{eq:Lyap} satisfies 
\begin{multline} \label{eq:Lyapmaindt}
\partial_t V(\boldsymbol{u}) \le  -\sum_{i=1}^3k_i\iint_\Omega  \bigg( \frac{ C(\Omega)}{Re} u_i^2   +  U_j  u_i \partial_{x_j} u_i \\  +u_j u_i  \partial_{x_j} U_i  -  u_i F_{ij} u_j \bigg) \,\, \mathrm{d}\Omega,
\end{multline}
where $C>0$.
\end{prop}

\begin{proof}
 The time derivative of  Lyapunov functional~\eqref{eq:Lyap} along the solutions of~\eqref{eq:mainNSEin} can be computed as
\begin{multline} \label{eq:Lyapdt}
\partial_t V(\boldsymbol{u}) = \sum_{i=1}^3\iint_\Omega k_i \bigg( \frac{1}{Re} u_i\nabla^2 u_i - u_j  u_i \partial_{x_j} u_i \\- U_j  u_i \partial_{x_j} u_i - u_j u_i  \partial_{x_j} U_i - u_i \partial_{x_i} p + u_i F_{ij} u_j +u_id_i\bigg) \,\, \mathrm{d}\Omega.
\end{multline}
Consider  $\iint_\Omega  k_i u_j  u_i \partial_{x_j} u_i\,\, \mathrm{d}\Omega$. Using the  boundary conditions, integration by parts and the incompressibility condition $ \partial_{x_j} u_j = 0$, we obtain
\begin{multline}
\iint_\Omega  k_i u_j  u_i \partial_{x_j} u_i\,\, \mathrm{d}\Omega = \frac{1}{2}\int_{ \Omega_i}k_i u_j u_i^2|_{\partial \Omega_j} \,\, \mathrm{d}x_i \\ - \frac{1}{2}\iint_\Omega  k_i  u_i^2 \left( \partial_{x_j}  u_j  \right) \,\, \mathrm{d}\Omega = 0.
\end{multline}
At this point, consider the pressure terms $\iint_\Omega k_i u_i \partial_{x_i} p \,\, \mathrm{d}\Omega$. Without loss of generality, we consider invariance in $x_1$, which yields
\begin{multline} \label{fsfstttd}
\iint_\Omega \left(k_2 u_2 \partial_{x_2} p + k_3 u_3 \partial_{x_3} p\right) \,\, \mathrm{d}\Omega \\ = \int_{\Omega_3} (k_2 u_2 p)|_{\partial\Omega_2} \,\, \mathrm{d}x_3 + \int_{\Omega_2}(k_3 u_3 p)|_{\partial\Omega_3} \,\, \mathrm{d}x_2\\-\iint_\Omega \left(k_2 \partial_{x_2} u_2  p + k_3 \partial_{x_3}  u_3 p\right) \,\, \mathrm{d}\Omega \\ = -\iint_\Omega \left(k_2 \partial_{x_2} u_2 + k_3 \partial_{x_3}  u_3 \right) p \,\, \mathrm{d}\Omega,
\end{multline}
where in the first equality above integration by parts and in the second inequality     the boundary conditions are used. Then, if $k_2=k_3$, using the incompressibility condition $\partial_{x_2} u_2 +  \partial_{x_3}  u_3=0$,  \eqref{fsfstttd} equals zero. Therefore, the time derivative of the Lyapunov/storage functional~\eqref{eq:Lyapdt} is modified to 
\begin{multline} \label{eq:Lyapmaindt1}
\partial_t V(\boldsymbol{u}) = \sum_{i=1}^3\iint_\Omega k_i \bigg( \frac{1}{Re} u_i\nabla^2 u_i  - U_j  u_i \partial_{x_j} u_i \\- u_j u_i  \partial_{x_j} U_i  + u_i F_{ij} u_j +u_id_i\bigg) \,\, \mathrm{d}\Omega.
\end{multline} 
Integrating by parts the $u_i\nabla^2 u_i$ term and using the   boundary conditions, we get
\begin{multline} \label{eq:Lyapmaindt2}
\partial_t V(\boldsymbol{u}) = \sum_{i=1}^3\iint_\Omega k_i \bigg( \frac{1}{Re} (\partial_{x_i}u_i)^2 - U_j  u_i \partial_{x_j} u_i \\- u_j u_i  \partial_{x_j} U_i  + u_i F_{ij} u_j + u_id_i \bigg) \,\, \mathrm{d}\Omega.
\end{multline} 
 Applying Poincar\'e inequality (Lemma 1 in Appendix~A) to~\eqref{eq:Lyapmaindt2}, we obtain \eqref{eq:Lyapmaindt}.
\end{proof}

\begin{rem}
A special case of~\eqref{eq:Lyap} was used in~\cite{JH71} to study the stability of viscous fluid flows. 
\end{rem}

\begin{rem}
In the sequel, we use   structure \eqref{eq:Lyap}  as a Lyapunov functional when studying stability and as a storage functional when studying input-to-state properties. 
\end{rem}

\begin{rem}
There are several estimates for the optimal  Poincar\'e constant. The optimal constant we use in this paper is 
\begin{equation}
C(\Omega) = \frac{\pi^2}{D(\Omega)},
\end{equation}
where $D(\Omega)$ is the diameter of the domain $\Omega$~\cite{PW60}.
\end{rem}

The next corollary  proposes conditions under which  properties such as stability, input-state induced $\mathcal{L}^2$ bounds and ISS can be inferred for the flow described by~\eqref{eq:mainNSEin}.

\begin{cor} \label{cor1}
Consider the flow described by~\eqref{eq:mainNSEin} subject to periodic or     no-slip boundary conditions $\boldsymbol{u}|_{\partial\Omega} = 0$. Assume the flow is invariant with respect to $x_m$, $m\in \{1,2,3\}$. Let $I=\{1,2,3\}-\{m\}$. If there exist positive constants $k_i$, $i=1,2,3$, with $k_i=k_j$, $i,j \in I$, positive scalars $\{\psi_i\}_{i\in \{1,2,3\}}$, $\{\eta\}_{i\in \{1,2,3\}}$, and $\sigma \in \mathcal{K}$ such that\\
\mbox{I)} when $\boldsymbol{d}\equiv 0$,
\begin{eqnarray} \label{eq:con1}
\sum_{i=1}^3k_i\iint_\Omega  \bigg( \frac{ C(\Omega)}{Re} u_i^2   +  U_j  u_i \partial_{x_j} u_i  \nonumber \\+u_j u_i  \partial_{x_j} U_i   -  u_i F_{ij} u_j \bigg) \,\, \mathrm{d}\Omega > 0
\end{eqnarray}
\mbox{II)}
\begin{multline}\label{eq:conL22}
\sum_{i=1}^3\iint_\Omega  \bigg( \left(\frac{ k_i C(\Omega)}{Re}-1\right) u_i^2   +  k_iU_j  u_i \partial_{x_j} u_i    \\+k_iu_j u_i  \partial_{x_j} U_i  -  k_iu_i F_{ij} u_j -k_iu_id_i + \eta_i^2 d_i^2\bigg) \,\, \mathrm{d}\Omega \ge 0
\end{multline}
\mbox{III)}
\begin{multline}\label{eq:conISS}
\sum_{i=1}^3 \iint_\Omega  \bigg( \left(\frac{ k_i C(\Omega)}{Re}-\psi_i k_i\right) u_i^2   +  k_iU_j  u_i \partial_{x_j} u_i    \\+k_iu_j u_i  \partial_{x_j} U_i  -  k_iu_i F_{ij} u_j \\-k_iu_id_i +\sigma(|d_1|,|d_2|,|d_3|)\bigg) \,\, \mathrm{d}\Omega \ge 0
\end{multline}
Then, \\
\mbox{I)} perturbation velocities given by~\eqref{eq:mainNSEin} are exponentially stable. Therefore, the flow converges to the laminar flow exponentially.\\
\mbox{II)} under zero perturbation initial conditions $\boldsymbol{u}(0,\mathrm{x})\equiv 0$, the induced $\mathcal{L}^2$ norm from inputs to perturbation velocities is bounded by $\eta_i$, $i\in \{1,2,3\}$ as in~\eqref{eq:L2}.\\
\mbox{III)} the perturbation velocities described by~\eqref{eq:mainNSEin} are ISS in the sense of~\eqref{eq:iss}.
\end{cor}
\begin{proof}
Each item is proven as follows. \\
\mbox{I)} Considering Lyapunov functional~\eqref{eq:Lyap}, inequality~\eqref{sda1} is satisfied with $c_1=\min_{i \in \{1,2,3\}} k_i$ and $c_2=\max_{i \in \{1,2,3\}} k_i$. Re-arranging the terms in~\eqref{eq:con1} gives
\begin{eqnarray}
-\sum_{i=1}^3 k_i\iint_\Omega  \bigg( \frac{ C(\Omega)}{Re} u_i^2   +  U_j  u_i \partial_{x_j} u_i   +u_j u_i  \partial_{x_j} U_i  \nonumber \\-  u_i F_{ij} u_j -u_id_i\bigg) \,\, \mathrm{d}\Omega < 0.
\end{eqnarray}
Then, from Proposition~\ref{prop1}, we infer that, for $d \equiv 0$, \mbox{$\partial_t V(\boldsymbol{u}) < 0$}. By continuity, we infer that there exists $c_3>0$ such that~\eqref{sddsf1} holds. Then, form Item I in Theorem 1, we infer that the perturbation velocities are exponentially stable. \\
\mbox{II)} Re-arranging terms in \eqref{eq:conL22} yields 
\begin{multline}\label{eq:conL2}
- \sum_{i=1}^3k_i\iint_\Omega  \bigg( \frac{  C(\Omega)}{Re} u_i^2   +  U_j  u_i \partial_{x_j} u_i    +u_j u_i  \partial_{x_j} U_i\\  -  u_i F_{ij} u_j -u_id_i \bigg) \,\, \mathrm{d}\Omega \\\le -\sum_{i=1}^3\iint_\Omega u_i^2 \,\, \mathrm{d}\Omega + \sum_{i=1}^3\iint_\Omega \eta_i^2 d_i^2\,\, \mathrm{d}\Omega
\end{multline}
Then, from \eqref{eq:Lyapmaindt} in Proposition~\ref{prop1}, we deduce that
$$
\partial_tV(\boldsymbol{u}) \le -\sum_{i=1}^3\iint_\Omega u_i^2 \,\, \mathrm{d}\Omega + \sum_{i=1}^3\iint_\Omega \eta_i^2 d_i^2\,\, \mathrm{d}\Omega.
$$
From Item II in Theorem 1, we infer that, under zero initial conditions, the perturbation velocities satisfy \eqref{eq:L2}.\\
\mbox{III)} Adopting~\eqref{eq:Lyap} as a storage functional,~\eqref{e11} is satisfied with $\beta_1(\cdot)=\min_{i\in \{1,2,3\}}k_i (\cdot)^2$ and $\beta_2(\cdot)=\max_{i\in \{1,2,3\}}k_i (\cdot)^2$. 
Re-arranging the terms in \eqref{eq:conISS}, we obtain
\begin{multline}
-\sum_{i=1}^3\iint_\Omega  \bigg( \frac{ k_i C(\Omega)}{Re}u_i^2   +  k_iU_j  u_i \partial_{x_j} u_i    +k_iu_j u_i  \partial_{x_j} U_i\\  -  k_iu_i F_{ij} u_j -k_iu_id_i \bigg) \,\, \mathrm{d}\Omega \\ \le -\sum_{i=1}^3\psi_i \iint_\Omega k_i u_i^2 \,\, \mathrm{d}\Omega + \iint_\Omega \sigma(|d_1|,|d_2|,|d_3|)\,\, \mathrm{d}\Omega
\end{multline}
From \eqref{eq:Lyapmaindt} in Proposition~\ref{prop1}, it follows that 
\begin{multline}
\partial_t V(\boldsymbol{u})  \le - \psi V(\boldsymbol{u}) + \iint_\Omega \sigma(|d_1|,|d_2|,|d_3|)\,\, \mathrm{d}\Omega,
\end{multline}
with $\psi= \min_{i\in\{1,2,3\}} \psi_i$. Then, from Item III in Theorem~1, we infer that the perturbation velocities satisfy the ISS property~\eqref{eq:iss}.
\end{proof}

%
 
 \section{Convex Formulation for Streamwise Constant Flows} \label{sec:convex}
 
 To present a convex method for checking the conditions in Corollary~\ref{cor1}, we restrict our attention to streamwise constant flows in $x_m$-direction with laminar flow $\boldsymbol{U} = U_m(\mathrm{x}) \overrightarrow{e}_m$. 
  
 \begin{cor} \label{LMIcor}
 Consider the perturbation dynamics given by~\eqref{eq:mainNSEin}. Assume streamwise invariance in $x_m$-direction with laminar flow $\boldsymbol{U} = U_m(\mathrm{x}) \overrightarrow{e}_m$ where $m \in \{1,2,3\}$. Let $I=\{1,2,3\}-\{m\}$. If there exist positive constants $\{k_l\}_{l \in \{1,2,3\}}$ with $k_p=k_q$, $p,q\in I$, $\{\eta_l\}_{l \in \{1,2,3\}}$,  $\{\psi_l\}_{l \in \{1,2,3\}}$, and functions $\{\sigma_l\}_{l \in \{1,2,3\}}$ such that \\
  \begin{figure*}[!t]
{
 \begin{equation} \label{eq:Mmat}
M(\mathrm{x}) = \begin{bmatrix} 
 \left(\frac{C}{Re}-F_{mm} \right)k_m &  \frac{k_m(\partial_{x_j}U_m(\mathrm{x})-F_{mj})-k_jF{jm}}{2}  & \frac{k_m(\partial_{x_i}U_m(\mathrm{x})-F_{mi})-k_iF{im}}{2}   \\  \frac{k_m(\partial_{x_j}U_m(\mathrm{x})-F_{mj})-k_jF{jm}}{2} &  \left(\frac{C}{Re}-F_{jj} \right)k_j & -\frac{k_j F_{jm}}{2} \\ \frac{k_m(\partial_{x_i}U_m(\mathrm{x})-F_{mi})-k_iF{im}}{2} & -\frac{k_j F_{jm}}{2}  & \left(\frac{C}{Re}-F_{ii} \right)k_i
 \end{bmatrix} \ge 0,~i,j \in I, i\neq j,~\mathrm{x}\in\Omega.
 \end{equation}
\hrulefill
\vspace*{4pt}}
\end{figure*}
 \mbox{I)} \eqref{eq:Mmat} holds,\\
 \mbox{II)} 
 \begin{multline} \label{eq:Nmat}
N(\mathrm{x}) =\\ \begin{pmat}[{..|}] 
~ & ~  & ~ & -\frac{k_m}{2} & 0 & 0\cr 
~ & M(\mathrm{x})-\mathrm{I}_{3\times 3} & ~ & 0 & -\frac{k_j}{2} & 0 \cr 
~ & ~& ~ & 0 & 0 & -\frac{k_i}{2} \cr\-
-\frac{k_m}{2} & 0 & 0  & \eta_{m}^2 & 0 & 0 \cr
0 & -\frac{k_j}{2} & 0  & 0 & \eta_{i}^2  & 0   \cr
0 & 0 & -\frac{k_i}{2}  & 0 & 0 & \eta_{j}^2  \cr
\end{pmat} \ge 0,
\end{multline}
for $i,j \in I, i\neq j$ and $\mathrm{x}\in \Omega$,\\
\mbox{III)} $\sigma_l(\mathrm{x}) \ge 0,~\mathrm{x} \in \Omega$, $l \in \{1,2,3\}$ and 
\begin{multline}\label{eq:Pmat}
P(\mathrm{x}) =\\\begin{pmat} [{..|}]  ~ & ~  & ~ & -\frac{k_m}{2} & 0 & 0 \cr
 ~ & M(\mathrm{x})-Q & ~ & 0 & -\frac{k_j}{2} & 0 \cr
~ & ~ & ~ & 0 & 0 & -\frac{k_i}{2} \cr\-
 -\frac{k_m}{2} & 0 & 0  & \sigma_{m}(\mathrm{x}) & 0 & 0 \cr
0 & -\frac{k_j}{2} & 0  & 0 & \sigma_{j}(\mathrm{x})  & 0   \cr
0 & 0 & -\frac{k_i}{2}  & 0 & 0 & \sigma_{i}(\mathrm{x}) \cr \end{pmat} \ge 0,
\end{multline}
for  $i,j \in I, i\neq j$ and $\mathrm{x}\in \Omega$, where $Q=\left[\begin{smallmatrix}\psi_mk_m & 0 & 0\\0&\psi_jk_j&0\\0&0&\psi_i k_i\end{smallmatrix} \right]$. Then, it follows that \\
\mbox{I)} the perturbation velocities are exponentially stable,\\
\mbox{II)} subject to zero initial conditions, the induced  $\mathcal{L}^2$ norm from inputs to perturbation velocities is bounded by $\eta_i$, \mbox{$i= 1,2,3$} as in~\eqref{eq:L2},\\
\mbox{III)} the perturbation velocities are ISS in the sense of~\eqref{eq:iss} with $\sigma(|\boldsymbol{d}|) =  \sum_{i=1}^3 \sigma_i(\mathrm{x}) d_i^2$.
  \end{cor}
  \begin{proof} 
  The proof is straightforward and  follows from computing conditions~\eqref{eq:con1},~\eqref{eq:conL2}, and~\eqref{eq:conISS} considering $x_m$-invariance,  the laminar flow $\boldsymbol{U} = U_m \overrightarrow{e}_m$, and $\sigma(|\boldsymbol{d}|) =  \sum_{i=1}^3 \sigma_i(\mathrm{x}) d_i^2$. Since the flow is $x_m$-invariant and  the laminar flow is given by $\boldsymbol{U} = U_m \overrightarrow{e}_m$, $U_j  u_i \partial_{x_j} u_i=0$, $i=1,2,3$. \\
  \mbox{I)} Inequality~\eqref{eq:con1} is given by
  \begin{multline} 
\mathcal{A} =\\ \iint_\Omega  \bigg( \left(\frac{ C(\Omega)}{Re}-F_{ii}\right)k_i u_i^2  - u_i(k_iF_{ij})u_j -u_i(k_iF_{im})u_m \\
+\left(\frac{ C(\Omega)}{Re}-F_{jj}\right)k_j u_j^2  - u_j(k_jF_{ji})u_i -u_j(k_jF_{jm})u_m \\
\left(\frac{ C(\Omega)}{Re}-F_{mm}\right)k_m u_m^2  + u_m(\partial_{x_i}U_m -F_{mi})u_i \\+ u_m(\partial_{x_j}U_m - F_{mj})
 \bigg) \,\, \mathrm{d}\Omega \ge 0 
\end{multline}
for $i,j \in I$, $i\neq j$, which can be rewritten as
\begin{equation} \label{p2323}
\iint_\Omega \left[ \begin{smallmatrix} u_m \\ u_j \\ u_i  \end{smallmatrix} \right]^\prime M(\mathrm{x}) \left[ \begin{smallmatrix} u_m \\ u_j \\ u_i  \end{smallmatrix} \right] \,\, \mathrm{d}\Omega \ge 0.
\end{equation}
with $M(\mathrm{x})$ given in~\eqref{eq:Mmat}. Therefore, if \eqref{eq:Mmat} is satisfied, \eqref{p2323} also holds and from Item I in Corollary~\ref{cor1} we infer that the perturbation velocities are exponentially stable.\\
\mbox{II)} Inequality~\eqref{eq:conL2} is changed to 
\begin{multline}
\mathcal{A} +\iint_\Omega (k_iu_id_i+k_ju_jd_j+k_mu_md_m)\,\, \mathrm{d}\Omega\\-\iint_\Omega (u_i^2+u_j^2+u_m^2)\,\, \mathrm{d}\Omega \\+ \iint_\Omega (\eta_i^2d_i^2+\eta_j^2d_j^2+\eta_m^2d_m^2)\,\, \mathrm{d}\Omega \ge 0,
\end{multline}
for $i,j \in I$, $i\neq j$, which can be rewritten as
\begin{equation} \label{jhjlhljll}
\iint_\Omega \left[ \begin{smallmatrix} u_m \\ u_j \\ u_i \\ d_m \\ d_j \\ d_i  \end{smallmatrix} \right]^\prime N(\mathrm{x}) \left[ \begin{smallmatrix} u_m \\ u_j \\ u_i  \\ d_m \\ d_j \\ d_i\end{smallmatrix} \right] \,\, \mathrm{d}\Omega \ge 0,
\end{equation}
where $N$ is defined in~\eqref{eq:Nmat}. Consequently, if \eqref{eq:Nmat}  is satisfied for all $\mathrm{x}\in \Omega$, \eqref{jhjlhljll}  holds and from Item II in Corollary~\ref{cor1} we infer that, subject to zero initial conditions, the induced  $\mathcal{L}^2$ norm from inputs to perturbation velocities is bounded by $\eta_i$, \mbox{$i= 1,2,3$} as in~\eqref{eq:L2}.\\\
\mbox{III)} The proof follows the same lines as the proof of Item II above.
  \end{proof}
  
  In the case that $U_m(\mathrm{x})$ is a polynomial in $\mathrm{x}$, inequalities~\eqref{eq:Mmat},~\eqref{eq:Nmat}, and~\eqref{eq:Pmat} are polynomial matrix inequalities that should be checked for all $\mathrm{x} \in \Omega$. If the set $\Omega$ is a semi-algebraic set then these inequalities can be cast as a sum-of-squares (SOS) program (see Appendix B) by using Putinar's Positivstellensatz theorem~\cite[Theorem 2.14]{Las09}.

  \begin{rem}
 In order to find upper-bounds on the induced $\mathcal{L}^2$-norm from the body forces $(d_1,d_2,d_3)$ to the perturbation velocities $\boldsymbol{u}$, we solve the following optimization problem
\begin{eqnarray}
&\min_{k_i,k_j} (\eta_1^2+\eta_2^2+\eta_3^2) & \nonumber \\
&subject~to~N(\mathrm{x})\ge 0,~k_i,k_j>0,~i,j \in I.&
\end{eqnarray}
  \end{rem}
  
   In the next section, we consider the analysis of the rotating Couette flow, which illustrate the proposed results.
  
\begin{figure}[!t]

\centering{
                \includegraphics[width=6.29cm]{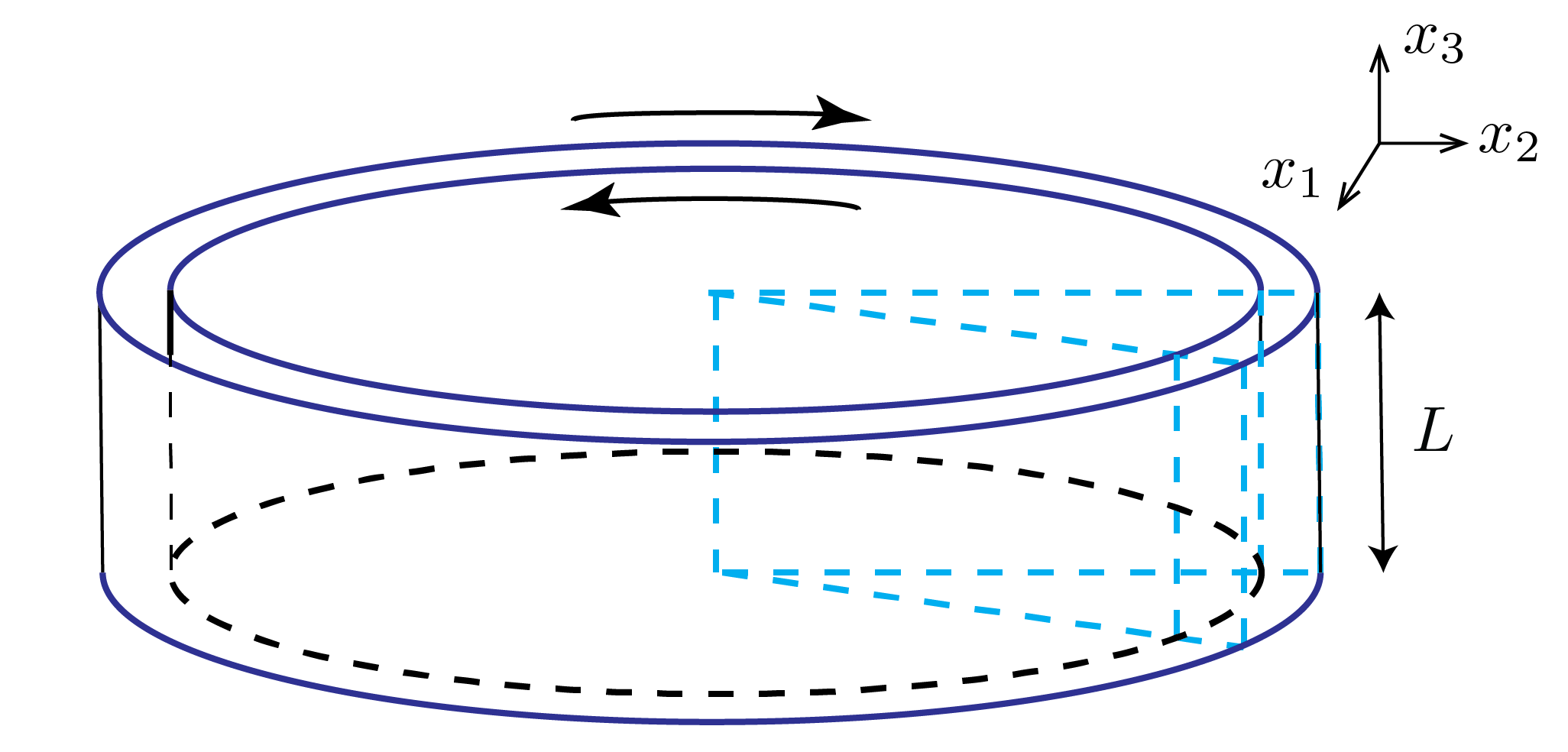}\\
                \includegraphics[width=6.29cm]{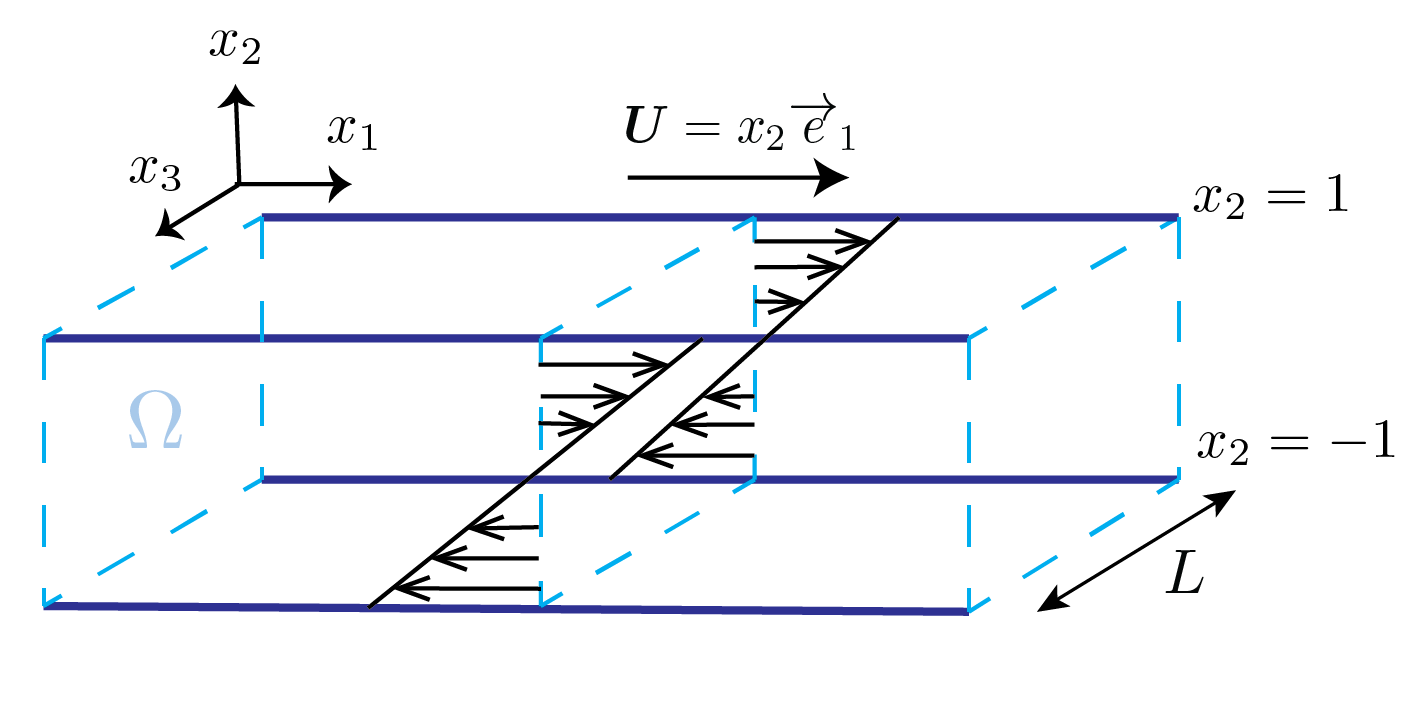}}

        \caption{Schematic of the rotating Couette flow geometry.}\label{fig1}
\end{figure}

\section{Example: Rotating Couette Flow} \label{sec:NR}

We consider the  flow of viscous fluid between two co-axial cylinders, where the gap between the cylinders is much smaller than  their radii. In this setting, the flow can be schematically illustrated as in Figure~\ref{fig1}. The axis of rotation is parallel to $x_3$-axis and the circumferential direction corresponds to $x_1$-axis.  Then, the dynamics of the perturbation velocities is  described by~\eqref{eq:mainNS}. The flow is assumed to be invariant with respect to $x_1$ \mbox{($\partial_{x_1}=0$)} and periodic   in $x_3$ with period $L$. Therefore, $\Omega = \left\{ (x_2,x_3) \mid (x_2,x_3) \in [-1,1]\times [0,L] \right\}$. The laminar flow is given by $\boldsymbol{U}=(x_2,0,0)'=x_2\overrightarrow{e}_1$ and $P=P_0$. In~addition, 
$$
F = \begin{bmatrix} 0 & Ro & 0 \\ -Ro & 0 & 0 \\ 0 & 0 & 0  \end{bmatrix},
$$
where $Ro \in [0,1]$ is a parameter representing the Coriolis force\footnote{ That is,  $Ro=0$ ($Ro=1$) corresponds to the case where only the outer (inner) cylinder is rotating and $Ro=0.5$ is the case where both cylinders are rotating with the same velocity but in opposite direction.}. We consider  no-slip boundary conditions \mbox{$\boldsymbol{u}|_{x_2=-1}^1 = 0$} and  $\boldsymbol{u}(t,x_2,x_3)=\boldsymbol{u}(t,x_2,x_3+L)$. The Poincar\'e constant is then given by $C=\frac{\pi^2}{L^2+2^2}$.


Notice that the cases $Ro=0,1$ correspond to the Couette flow. Thus, the obtained results for rotating Couette flow can be applied to the Couette flow in special cases, as well. We are interested in finding estimates of the critical Reynolds number $Re_C$ using the following Lyapunov functional 
$$
V(u) = \int_0^{L} \int_{-1}^1 \left[\begin{smallmatrix} u_1 \\ u_2 \\ u_3 \end{smallmatrix}\right]^\prime \left[\begin{smallmatrix} k_1 & 0 & 0 \\ 0 & k_2 & 0 \\ 0 & 0 & k_2 \end{smallmatrix} \right]\left[\begin{smallmatrix} u_1 \\ u_2 \\ u_3 \end{smallmatrix} \right] \,\, \mathrm{d}x_2\mathrm{d}x_3,
$$
which is the same as Lyapunov functional~\eqref{eq:Lyap} considering invariance with respect to $x_1$. 

For stability analysis, we need to check inequality~\eqref{eq:Mmat} according to Item I in Corollary~\ref{LMIcor}. For this flow ($m=1,j=2,i=3$), we have
\begin{equation} \label{sddfsdf}
M = \begin{bmatrix} \frac{k_1C}{Re} &    \frac{k_2Ro - k_1(Ro-1)}{2} & 0 \\ \frac{k_2Ro - k_1(Ro-1)}{2} & \frac{k_2C}{Re}  & 0 \\
0 & 0 & \frac{k_2C}{Re}
\end{bmatrix} \ge 0
\end{equation}
This is a linear matrix inequality (LMI) feasibility problem with decision variables $k_1,k_2>0$.

To find estimates of $Re_C$ in the case of Couette Flow $Ro=0$, applying Schur complement theorem~\cite[p. 650]{BV04} to \eqref{sddfsdf},   we have
\begin{equation*}
\frac{k_1C}{Re} - \left(\frac{ k_1}{2} \right)^2 \left( \frac{Re}{k_2C} \right)\ge 0,~\frac{k_2C}{Re}\ge0,
\end{equation*}
which yields the inequality\footnote{For $Ro=1$, we can similarly obtain $\frac{k_1}{k_2} \ge \left( \frac{Re}{2C}\right)^2$.}
\begin{equation} \label{eq:satbcond}
\frac{k_2}{k_1} \ge \left( \frac{Re}{2C}\right)^2.
\end{equation}
This implies that the Couette flow is stable for all $Re$. Hence, for Couette flow, $Re_C=\infty$  obtained using Lyapunov functional~\eqref{eq:Lyap} coincides with linear stability limit \mbox{$Re_L = \infty$}~\cite{Romanov73}. 

\begin{figure}[!t]

\centerline{
                \includegraphics[scale=.3]{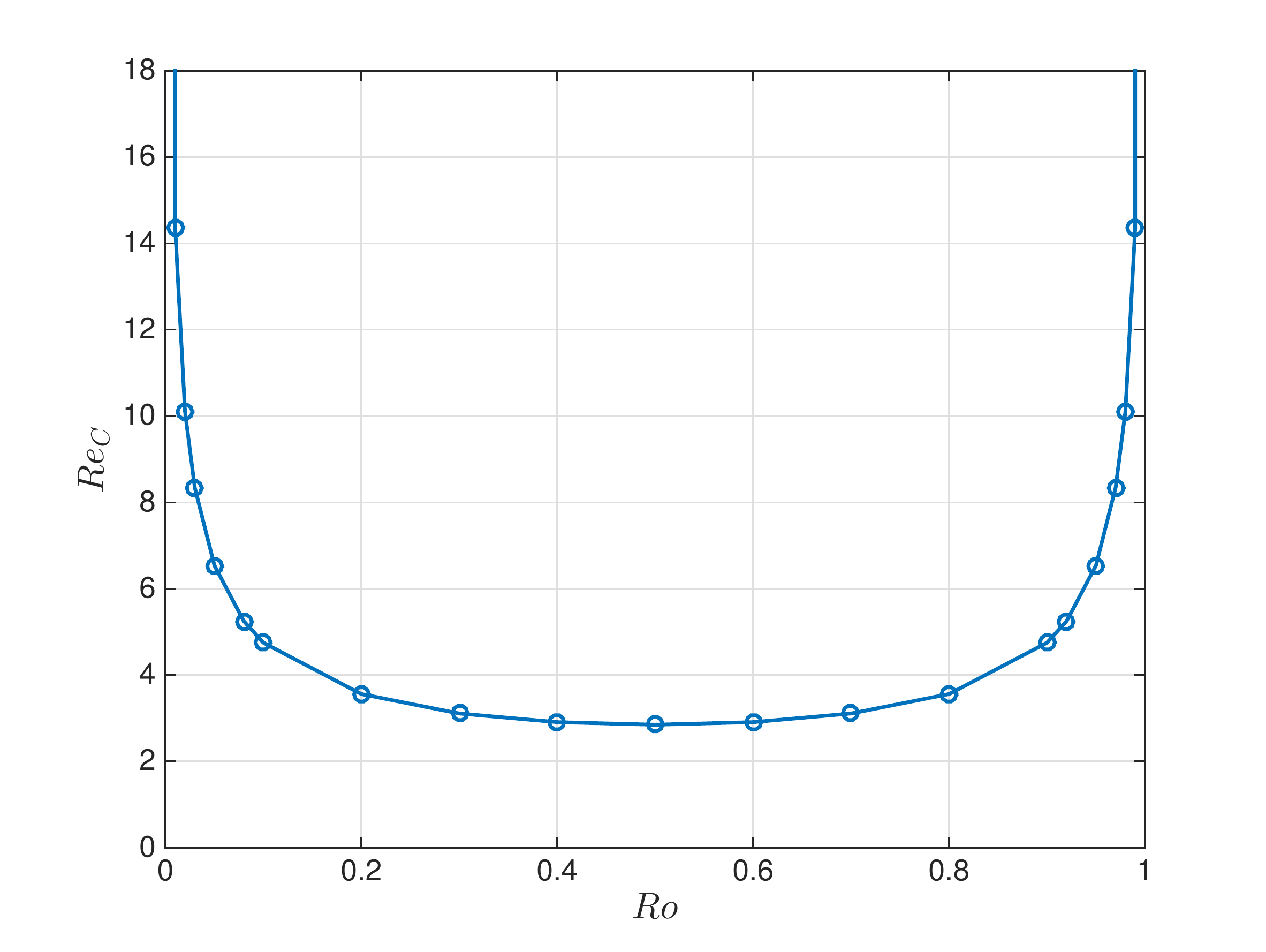}
}
        \caption{Estimated critical Reynolds numbers $Re$ in terms of $Ro$ for rotating Couette flow.}\label{fig2}
\end{figure}

\begin{figure*}[!t]

\centerline{
                \includegraphics[scale=.255]{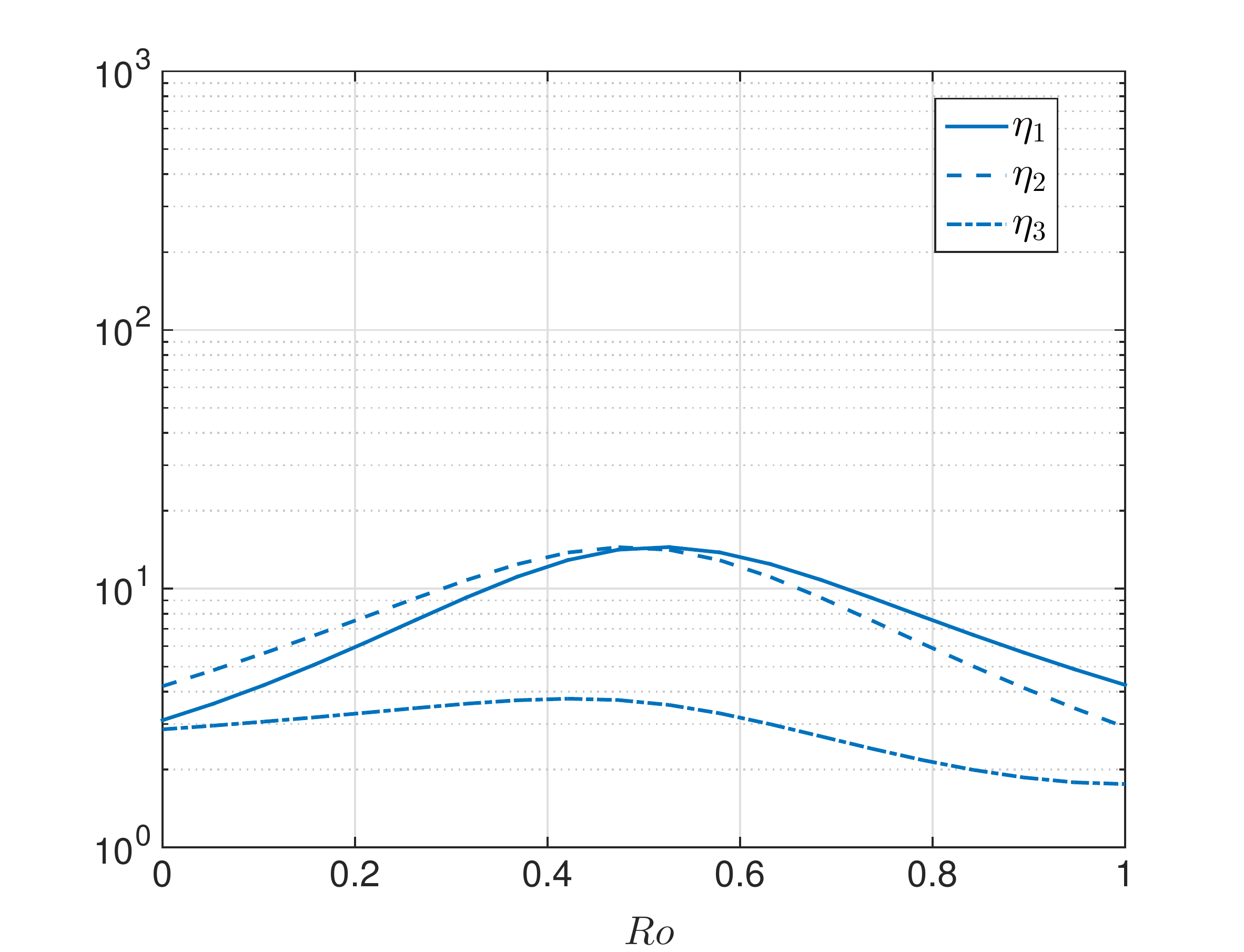}
                \includegraphics[scale=.255]{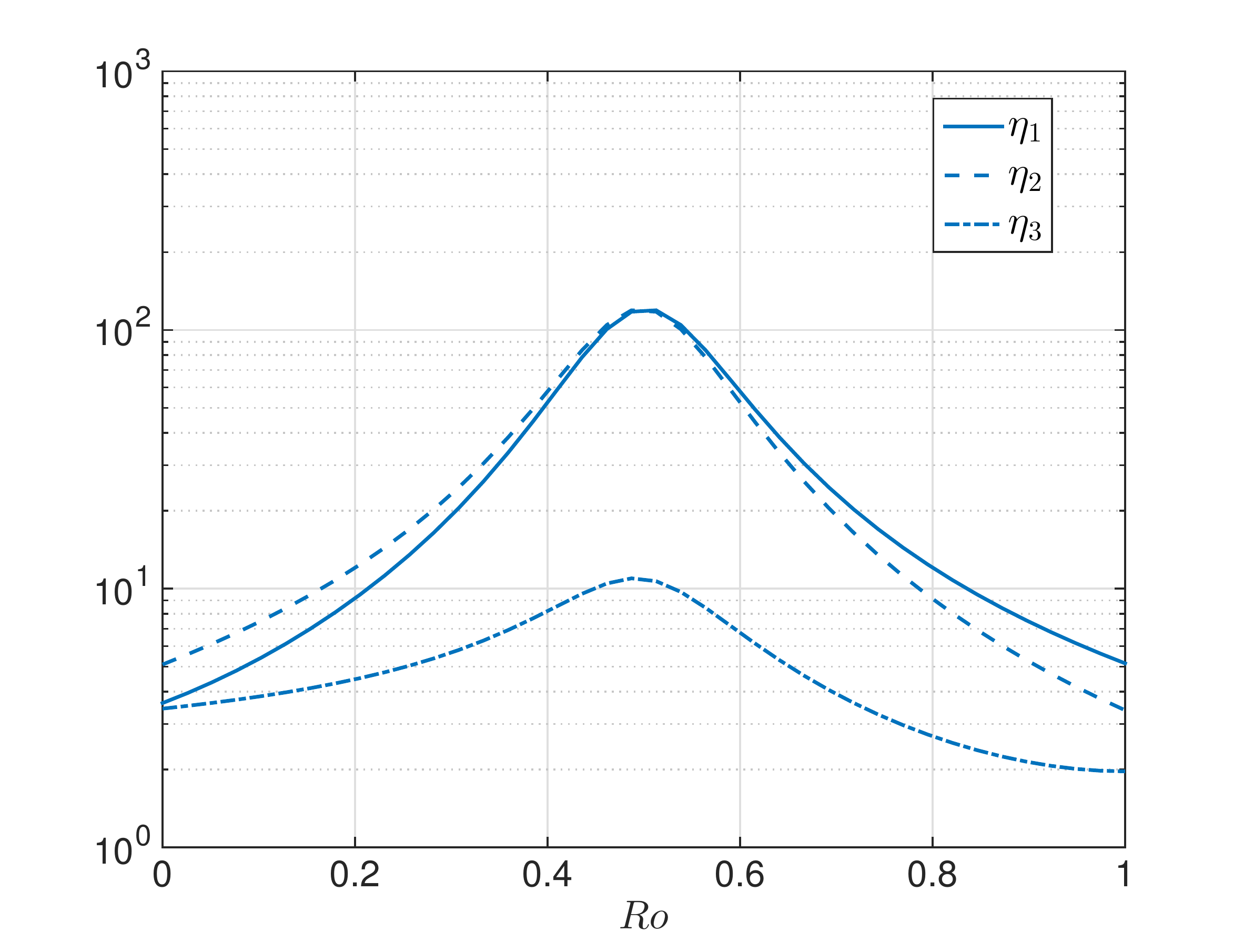}
                \includegraphics[scale=.255]{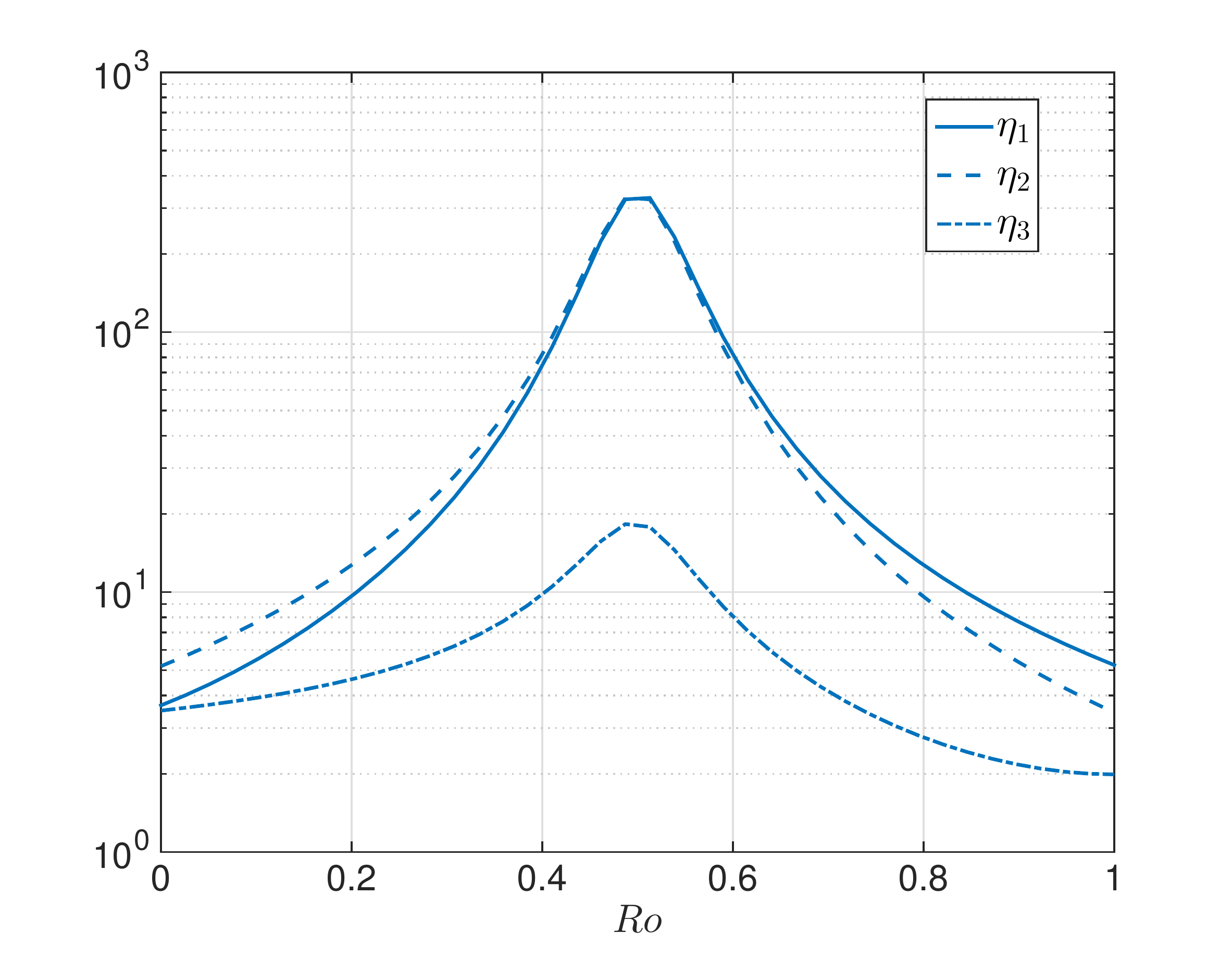}                              
}
        \caption{Upper bounds on induced $\mathcal{L}^2$-norms from $\boldsymbol{d}$ to perturbation velocities $\boldsymbol{u}$ of rotating Couette flow for different Reynolds numbers: $Re=2$ (left), $Re=2.8$ (middle), and $Re=2.83$ (right).}\label{fig3}
\end{figure*}

Let $L = \pi$. Figure \ref{fig2} illustrates the estimated critical Reynolds numbers $Re_C$ as a function of $Ro$ obtained from solving the LMI \eqref{sddfsdf} and performing a line search over $Re$. Notice that for the cases $Ro=0,1$ the flow is stable for all Reynolds numbers.

For induced $\mathcal{L}^2$-norm analysis, we apply inequality~\eqref{eq:Nmat} which for this particular flow is given by the following LMI

\begin{eqnarray}
N = \begin{pmat}[{..|}] 
~ & ~  & ~ & -\frac{k_1}{2} & 0 & 0\cr 
~ & M-\mathrm{I}_{3\times 3} & ~ & 0 & -\frac{k_2}{2} & 0 \cr 
~ & ~& ~ & 0 & 0 & -\frac{k_2}{2} \cr\-
-\frac{k_1}{2} & 0 & 0  & \eta_{1}^2 & 0 & 0 \cr
0 & -\frac{k_2}{2} & 0  & 0 & \eta_{2}^2  & 0   \cr
0 & 0 & -\frac{k_2}{2}  & 0 & 0 & \eta_{3}^2  \cr
\end{pmat} \ge 0 \nonumber
\end{eqnarray}
with $M$ as in~\eqref{sddfsdf}. 

Figure~\ref{fig3} depicts the obtained results for three different Reynolds numbers. As the Reynolds number approaches the estimated $Re_C$ for $Ro=0.5$, the upper-bounds on the induced $\mathcal{L}^2$-norm from the body forces $\boldsymbol{d}$ to perturbation velocities $\boldsymbol{u}$ increases dramatically.

The obtained upper-bounds on the induced $\mathcal{L}^2$-norm for Couette flow $Ro=0$, are also given in Figure~\ref{fig4}. Since the flow is stable for all Reynolds numbers, the induced $\mathcal{L}^2$-norms keep increasing with Reynolds number. The obtained upper-bounds depicted in Figure~\ref{fig4} are consistent with Corollary 2 and Corollary 4 in~\cite{JB05} and Theorem 1 in~\cite{BD01}, wherein it was demonstrated that $\eta_1^2 \propto O(Re)$, and $\eta_{2}^2,\eta_{3}^2 \propto O(Re^3)$ for Couette flow.

In order to check the ISS property, we check inequality~\eqref{eq:Pmat} from Corollary~\ref{LMIcor}  for the rotating Couette flow under study, i.e., 
\begin{equation}
P =\begin{pmat} [{..|}]  ~ & ~  & ~ & -\frac{k_1}{2} & 0 & 0 \cr
 ~ & M-Q & ~ & 0 & -\frac{k_2}{2} & 0 \cr
~ & ~ & ~ & 0 & 0 & -\frac{k_2}{2} \cr\-
 -\frac{k_1}{2} & 0 & 0  & \sigma_{1} & 0 & 0 \cr
0 & -\frac{k_2}{2} & 0  & 0 & \sigma_{2}  & 0   \cr
0 & 0 & -\frac{k_2}{2}  & 0 & 0 & \sigma_{3} \cr \end{pmat} \ge 0 \nonumber
\end{equation}
with $M$ given in~\eqref{sddfsdf} and $Q=\left[\begin{smallmatrix} k_1 \psi_1 & 0 & 0\\ 0 & k_2 \psi_2 & 0\\0& 0 & k_2 \psi_3 \end{smallmatrix}\right]$.  We fix $\psi_i = 10^{-4},~i=1,2,3$ and $L=2\pi$. Figure~\ref{figISS} depicts the maximum Reynolds number for which ISS certificates could be found $Re_{ISS}$ and stability critical Reynold's numbers $Re_C$ as a function of $Ro$. It appears that for $Ro \in (0,1)$ these two quantities coincide. However, for the case of Couette flow $Ro=0,1$, we obtain $Re_{ISS} = 316$ and $Re_C = \infty$. The quantity $Re_{ISS}=316$ is the closest estimate to the empirical Reynolds number $Re \approx 350$~\cite{TA92}.


\section{CONCLUSIONS AND FUTURE WORK} \label{sec:conclusions}

\subsection{Conclusions}

We studied stability and input-state properties of fluid flows with invariance in one direction. We formulated a class of appropriate Lyapunov/storage functionals for such flows. Conditions based on matrix inequalities are given for streamwise constant flows. When the laminar flow is given by a polynomial of spatial coordinates, the matrix inequalities can be checked using convex optimization. For illustration purposes, we applied the proposed method to study a model of rotating Couette flow.

\subsection{Future Work}

In this study, we considered flows in the Cartesian coordinate system. For many flows, like pipe Poiseuille flow, the coordinate system is naturally cylindrical. An extension of the results proposed in this paper to cylindrical coordinates is under study. 

In several scenarios in fluid mechanics, we are interested in a functional of the perturbation dynamics. For example, in the drag estimation problem, we are interested in estimating the functional of pressure over the surface of an airfoil. We are currently applying the methodology proposed in~\cite{AVP15} to address such problems in fluid mechanics.

\begin{figure}[!t]

\centerline{
                \includegraphics[scale=.3]{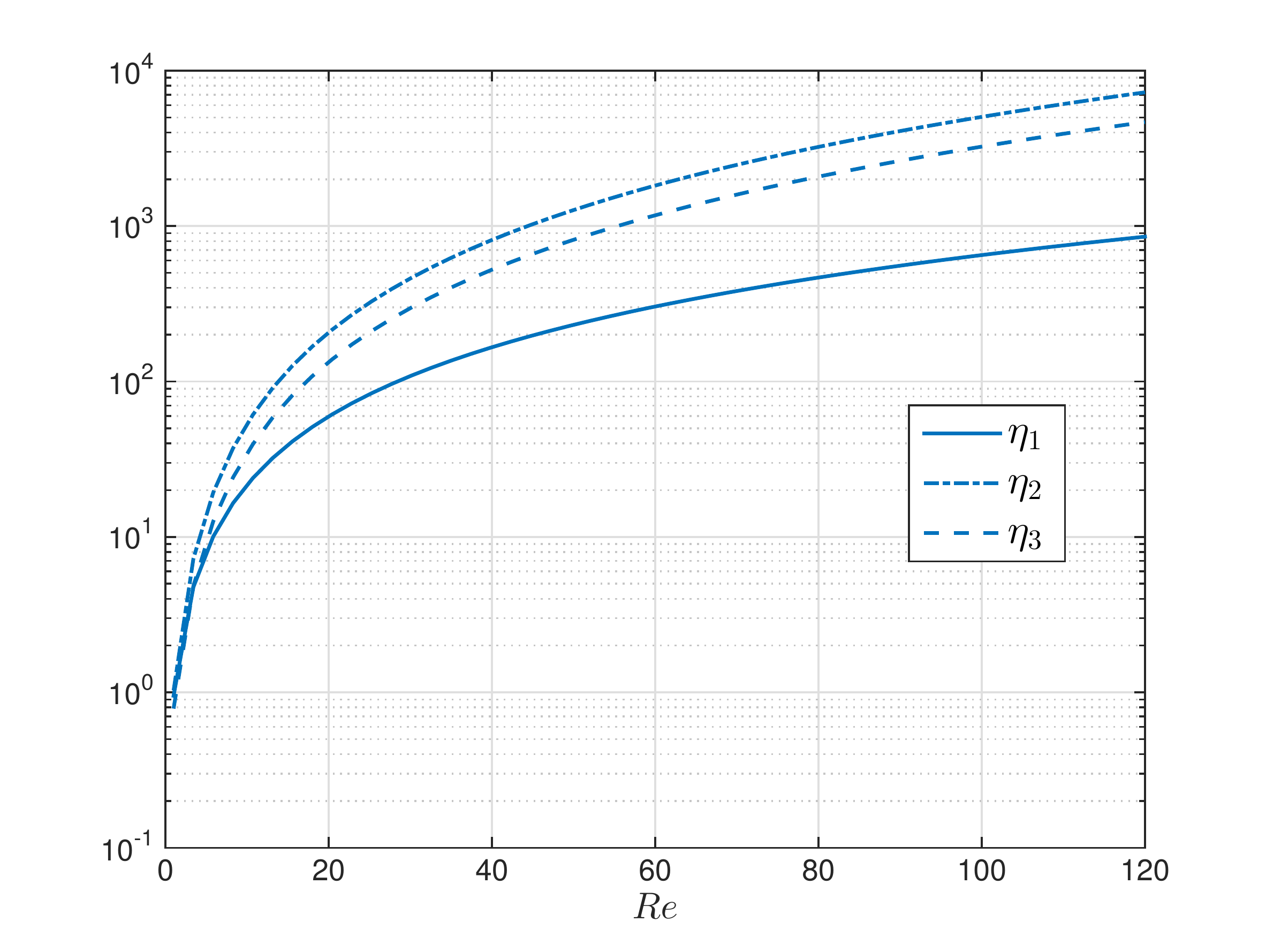}                           
}
        \caption{Upper bounds on induced $\mathcal{L}^2$-norms for perturbation velocities of Couette flow for different Reynolds numbers.}\label{fig4}
\end{figure}

\begin{figure}[!t]

\centerline{
                \includegraphics[scale=.3]{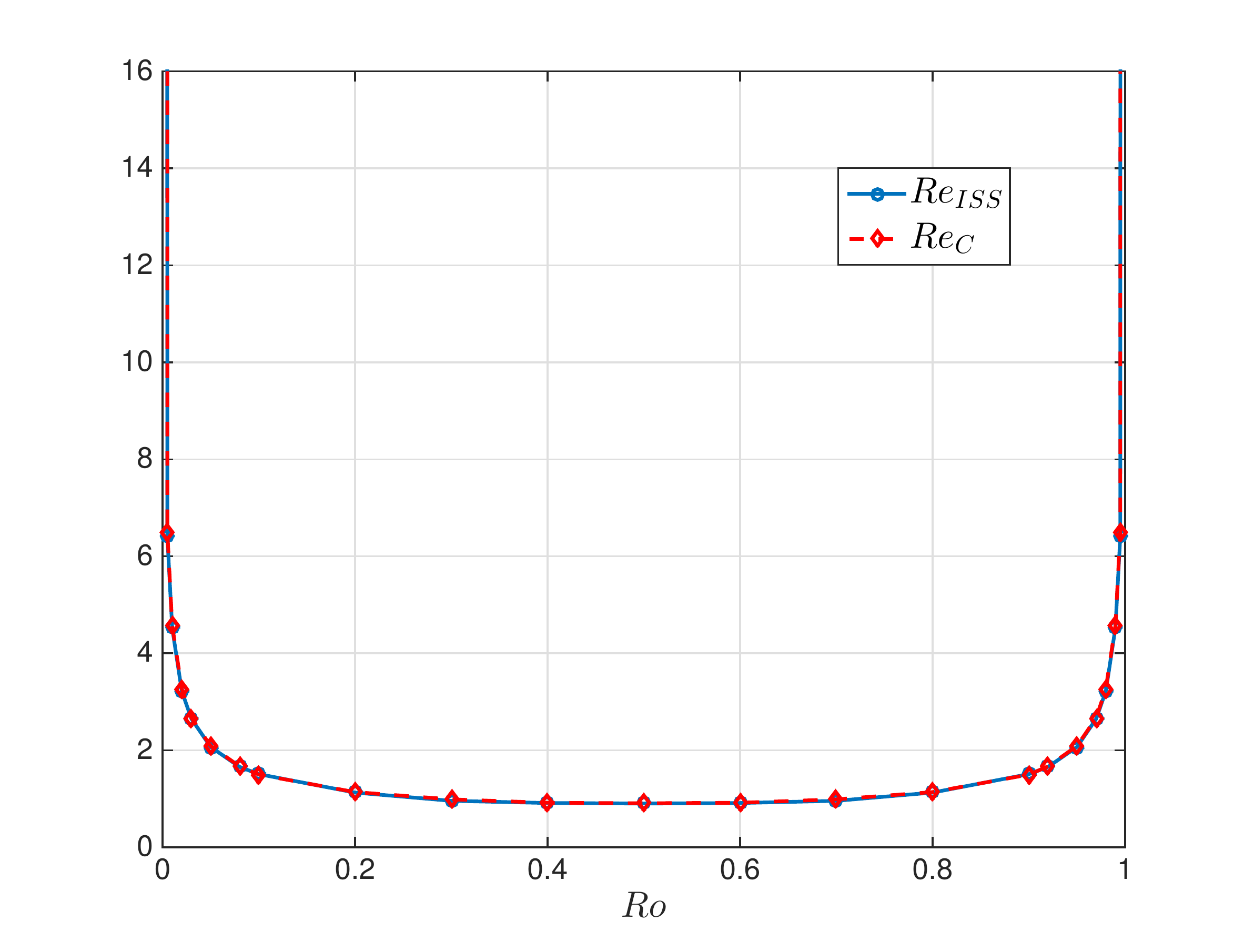}
}
        \caption{Estimated critical Reynolds numbers $Re_C$ and ISS Reynolds numbers $Re_{ISS}$ in terms of $Ro$.}\label{figISS}
\end{figure}

\section{ACKNOWLEDGMENTS}
The authors appreciate stimulating discussions by Prof. Sergei Chernyshenko from Department of Aeronautics, Imperial College London and Prof. Charles Doering, from Department of Mathematics, University of Michigan.


\bibliography{references}
\bibliographystyle{IEEEtran}


\appendix

\subsection{Poincar\'e Inequality}
\begin{lem}[\cite{PW60}]\label{inq:poincare}
Assume $\Omega \subset \mathbb{R}^2$ is a bounded, convex, Lipschitz domain with diameter $D$, and $\boldsymbol{u} \in \mathcal{C}^1(\Omega)$  with no-slip \mbox{$\boldsymbol{u}|_{\partial \Omega}=0$}  or periodic such that  $\iint_\Omega \boldsymbol{u} \,\, \mathrm{d}\Omega = 0$ boundary conditions. Then, the following inequality holds
\begin{equation*}
\frac{\pi}{D} \|  u \|_{\mathcal{L}^2_\Omega}  \le \| \nabla u \|_{\mathcal{L}^2_\Omega} .
\end{equation*}
\end{lem}

\subsection{Sum-of-Squares Programming}
Denote the ring of polynomials with real coefficients by $\mathcal{R}[x]$,  and the ring of polynomials with a sum-of-squares decomposition by $\Sigma[x] \subset \mathcal{R}[x]$. A polynomial $p(x)\in \Sigma[x]$ if $\exists p_i(x) \in \mathcal{R}[x]$, $i \in \{1, \ldots, n_d\}$ such that $p(x) = \sum_{i=1}^{n_d} p_i^2(x)$. Hence, $p(x)$ is clearly non-negative. The set of polynomials~$\{p_i\}_{i=1}^{n_d}$ is called \emph{SOS decomposition} of $p(x)$. The converse does not hold in general, that is, there exist non-negative polynomials which do not have an SOS decomposition~\cite{Par00}.  To test whether  an SOS decomposition exists for a given polynomial, one can solve an SDP (see~\cite{CLR95,Par00,CTVG99}). SOSTOOLS~\cite{sostools} is a software package for solving SOS programs.

\end{document}